\documentclass{amsart}
\usepackage[mathscr]{eucal}
\usepackage{graphicx}
\usepackage{amscd}
\usepackage{amsmath}
\usepackage{amsthm}
\usepackage{amsxtra}
\usepackage{calc}
\usepackage{amsfonts}
\usepackage{amssymb}
\usepackage{latexsym}
\usepackage{pdfsync}
\usepackage{amsopn}
\usepackage{mathrsfs}
\usepackage{accents}
\usepackage[all]{xy}
\SelectTips{cm}{}
\calclayout

\usepackage[colorlinks=true,linkcolor=red,citecolor=green,bookmarks=true]{hyperref}

\newcommand{\rt}{\rightarrow}
\newcommand{\lra}{{\xymatrix@C=3em{\ar @{>} [r] & }}}
\newcommand{\lrt}{\longrightarrow}

\newcommand{\llf}{\longleftarrow}

\def\db{\operatorname{\mathsf{D^b}}}

\newcommand{\pd}{{\rm{pd}}}

\def\ds{\operatorname{\mathsf{D_{sg}}}}

\newcommand{\si}{\mathsf{\Sigma}}
\newcommand{\syz}{\mathsf{\Omega}}

\newcommand{\st}{\stackrel}

\newcommand{\HT}{\mathsf{H}}

\newcommand{\Pair} {\mathsf{Pair}}

\newcommand{\DB} {\mathsf{DB}}

\newcommand{\Z}{\mathbb{Z} }

\newcommand{\G}{\mathcal{G} }
\newcommand{\umon}{\underline{\mathsf{Mon}}}

\newcommand{\n}{{\mathfrak{n}}}

\newcommand{\con}{\mathsf{cone}}

\newcommand{\cok}{{\rm{Coker}}}
\newcommand{\Ker}{{\rm{Ker}}}

\newcommand{\Tor}{\mathsf{Tor}}

\newcommand{\imm}{\rm{Im}}

\newcommand{\Hom}{{\mathsf{Hom}}}

\newcommand{\id}{{\mathsf{id}}}

\newcommand{\mon}{{\mathsf{Mon}}}

\newcommand{\gpd}{{\mathsf{Gpd}}}

\def\si{\operatorname{\mathsf{\Sigma}}}

\newcommand{\Ext}{\mathsf{{Ext}}}

\newcommand{\A}{\mathcal{{A}}}

\newcommand{\C}{\mathcal{C}}
\newcommand{\cp} {\mathcal{P}}
\newcommand{\E}{\mathcal{E}}

\newtheorem{theorem}{Theorem}[section]

\newtheorem{cor}[theorem]{Corollary}

\newtheorem{lemma}[theorem]{Lemma}
\newtheorem{prop}[theorem]{Proposition}

\theoremstyle{definition}
\newtheorem{example}[theorem]{Example}

\newtheorem{remark}[theorem]{Remark}

\newtheorem{s}[theorem]{}

\theoremstyle{plain}

\theoremstyle{definition}
\newtheorem{dfn}[theorem]{Definition}

\numberwithin{equation}{section}

\begin{document}

\title[stable category of monomorphisms between (Gorenstein) projective modules]
{The stable category of monomorphisms between (Gorenstein) projective modules with applications}

\author[Bahlekeh, Fotouhi, Hamlehdari and Salarian]{Abdolnaser Bahlekeh, Fahimeh Sadat Fotouhi, Mohammad Amin Hamlehdari and Shokrollah Salarian}



\address{Department of Mathematics, Gonbad-Kavous University, Postal Code:4971799151, Gonbad-e-Kavous, Iran}
\email{bahlekeh@gonbad.ac.ir}
\address{School of Mathematics, Institute for Research in Fundamental Science (IPM), P.O.Box: 19395-5746, Tehran, Iran}
\email{ffotouhi@ipm.ir}

\address{Department of Pure Mathematics, Faculty of Mathematics and Statistics, University of Isfahan, P.O.Box: 81746-73441, Isfahan, Iran }\email{amin.hamlehdari@sci.ui.ac.ir}

\address{Department of Pure Mathematics, Faculty of Mathematics and Statistics, University of Isfahan, P.O.Box: 81746-73441, Isfahan,
Iran and \\ School of Mathematics, Institute for Research in Fundamental Science (IPM), P.O.Box: 19395-5746, Tehran, Iran}
\email{Salarian@ipm.ir}


\subjclass[2020]{13D09, 18G80, 13H10, 13C60}

\keywords{monomorphism category, stable category, triangulated category, singularity category, the category of D-branes of type B, matrix factorization}

\thanks{The research of the second author was in part supported by a grant from IPM}

\begin{abstract}
Let $(S, \n)$ be a commutative noetherian local ring and let $\omega\in\n$ be non-zero divisor. This paper is concerned with the two categories of monomorphisms between finitely generated (Gorenstein) projective $S$-modules, such that their cokernels are annihilated by $\omega$. It is shown that these categories, which will be denoted by $\mon(\omega, \cp)$ and $\mon(\omega, \G)$, are both Frobenius categories with the same projective objects.
It is also proved that the stable category $\umon(\omega, \cp)$ is triangle equivalent to the category of D-branes of type B, $\DB(\omega)$, which has been introduced by Kontsevich and studied by Orlov. Moreover, it will be observed that the stable categories $\umon(\omega, \cp)$ and $\umon(\omega, \G)$ are closely related to the singularity category of the factor ring $R=S/({\omega)}$. Precisely, there is a fully faithful triangle functor from the stable category $\umon(\omega, \G)$ to $\ds(R)$, which is dense if and only if $R$ (and so $S$) are Gorenstein rings. Particularly, it is proved that the density of the restriction of this functor to $\umon(\omega, \cp)$, guarantees the regularity of the ring $S$.
\end{abstract}

\maketitle

\tableofcontents
\section{Introduction}
Assume that $(S, \n)$ is a commutative noetherian local ring, $\omega\in\n$ a non-zerodivisor and $R$ is the factor ring $S/{(\omega)}$. This paper deals with the following two subcategories of the monomorphism category of finitely generated $S$-modules:\\ (1) The category $\mon(\omega, \cp)$ consisting of all monomomorphisms $(P\st{f}\rt Q)$ with $P, Q$ projective $S$-modules and $\cok f$ is annihilated by $\omega$.\\ (2) The category $\mon(\omega, \G)$ consisting of all monomomorphisms $(G\st{f}\rt P)$ where $G$ is Gorenstien projective and $P$ is projective $S$-modules such that $\cok f$ is annihilated by $\omega$.

Our motivation to investigate these categories is their ability to describe some nice properties of the rings $R$ and $S$.
It is shown that these subcategories are both Frobenius with the same projective objects. {Namely, we show that these are exact categories in the sense of Quillen which have enough projectives and injectives, and whose projectives and injectives coincide, see Theorem \ref{frob} and Corollary \ref{s1}. So the stable categories $\umon(\omega, \cp)$ and $\umon(\omega, \G)$ will be triangulated. Our first main result reveals that the stable category $\umon(\omega, \cp)$ is triangle equivalent to the category D-branes of type B, $\DB(\omega)$. Indeed, we have the following theorem.
\begin{theorem}(1) The inclusion functor $i:\mon(\omega, \cp)\rt\mon(\omega, \G)$ induces a fully faithful triangle functor $\underline{i}:\umon(\omega, \cp)\rt\umon(\omega, \G)$.\\ (2) There is a triangle equivalence functor $F:\umon(\omega, \cp)\rt\DB(\omega)$.
\end{theorem}

}
A mathematical definition of $B$-branes ($D$-branes of type $B$) in affine Landau-Ginzburg models is proposed by M. Kontsevich \cite{kontsevich1994homological}. A Landau-Ginzburg model is a regular scheme $X$ and a regular function $W$ on $X$ such that the morphism $W: X\lra \mathbb{A}^1$ is flat. {A result due to Orlov \cite{orlov2003triangulated} indicates that the category $\DB(\omega)$ is nothing more than the homotopy category of matrix factorizations. We should remember that the notion of matrix factorizations appeared many years ago in the paper of Eisenbud \cite{eisenbud1980homological}, where he has used to study maximal Cohen-Macaulay modules over local rings, for more details see \ref{matrix}. It is shown by Buchweitz \cite{buchweitz1987maximal} and Orlov \cite[Theorem 3.9]{orlov2003triangulated} that the category $\DB(\omega)$ is triangle equivalent to the singularity category $\ds(R)$ of $R$, provided that $S$ is a regular ring. Furthermore, an analog of this result for more general complete intersections has been proved by Berg and Jorgensen in \cite{bergh2015complete}.

Our second main observation indicates that there is a tie connection between the aforementioned stable categories and the singularity category $\ds(R)$. Precisely, we prove the following theorem.
\begin{theorem}(1) The functor $T:\umon(\omega, \G)\rt\ds(R)$ sending each object $(G\st{f}\rt P)$ to $\cok f$, viewed as a stalk complex, is a fully faithful triangle functor. Moreover, $T$ is dense if and only if $R$ (and so, $S$) is Gorenstein. \\ (2) The fully faithful triangle functor $T\circ\underline{i}:\umon(\omega, \cp)\rt\ds(R)$ is dense if and only if $S$ is a regular ring.
\end{theorem}
The first statement is proved in Theorem \ref{gor} and the second one appears in Theorem \ref{reg1}.

We should remark that the singularity category $\ds(R)$ is by definition the Verdier quotient of the bounded derived category $\db(R)$ of $R$ by the perfect complexes. This notion has been introduced by Buchweitz \cite{buchweitz1987maximal}, and studied actively ever since the relation with mirror symmetry was found by Orlov \cite{orlov2003triangulated}.
}

{The contributions of the paper can be summarized as follows. We consider the monomorphism categories $\mon(\omega, \cp)$ and $\mon(\omega, \G)$ with the functors $\Pair(\omega)\st{F}\llf\mon(\omega, \cp)\st{i}\lrt\mon(\omega, \G)$, where $\Pair(\omega)$ is the category of pairs which is known to be a Frobenius category, see \cite{orlov2003triangulated}. It will be shown that $\mon(\omega, \cp)$ and $\mon(\omega, \G)$ are Frobenius categories, and so, their stable categories are triangulated. Moreover, the relationships of these triangulated categories with $\ds(R)$ and $\DB(\omega)$ are discussed. Indeed, we have the following diagram of triangle functors and equivalences:

{\footnotesize\[\xymatrix{\umon(\omega, \cp)\ar[rr]^{\underline{i}}\ar[dr]_{\underline{T}_1}^{\simeq}\ar[dd]_{F}^{\simeq} && \umon(\omega, \G)\ar[dd]^{T} & \\ & \underline{\Pair}(\omega)\ar[dl]_{\simeq}& \\ \DB(\omega)\ar[rr]~& & \ds(R),}\]}where $\underline{i}$ and $T$ are fully faithful triangle functors. Moreover, the density of $T$ (resp. $T\circ\underline{i}$) describes the Gorensteiness of $R$ (resp. regularity of $S$).

The study of monomorphism categories, known also as submodule categories, goes back to Birkhoff \cite{bir} in 1934, in which he initiated to classify the indecomposable objects of the submodule category of $\Z/{(p^n)}$, with $n\geq 2$ and $p$ a prime number. These categories provide a framework for addressing open problems in linear algebra using tools and results from homological algebra, combinatorics, and geometry. From this fact, the study of monomorphism categories has been the subject of many papers during the last decades, see for example \cite{ringel2008invariant, stable,ringel2014submodule,kussin2013nilpotent,xiong2014auslander,
asadollahi2022monomorphism,hafezi2022stable,hafezi2021subcategories,kosakowska2023abelian}
and references therein.}

As a convention throughout the paper, $(S, \n)$ is a commutative noetherian local ring, $\omega\in\n$ a non-zerodivisor, and $R$ is the factor ring $S/{(\omega)}$. Also, by a module, we mean a finitely generated module.

\section{The categories of monomorphisms between (Gorenstein) projective modules}
This section is devoted to investigating two kinds of subcategories of the monomorphism category of finitely generated $S$-modules. Precisely, we focus on monomorphisms between (Gorenstein) projective $S$-modules in which their cokernels are annihilated by $\omega$. It will be shown that these categories are Frobenius with the same projective objects. Let us begin this section with the following definition.

\begin{dfn}
By the category $\mon(\omega, \G)$, we mean a category that whose objects are those $S$-monomorphisms $(G\st{f}\rt P)$, where $G\in\G(S),$ $P\in\cp(S)$ and $\cok f$ is an $R$-module, {that is, $\cok f$ is annihilated by $\omega$}. Here $\G(S)$ (resp. $\cp(S)$) is the category of all finitely generated Gorenstein projective (resp. projective) $S$-modules. Moreover, a morphism $\psi=(\psi_1, \psi_0):(G\st{f}\rt P)\lrt (G'\st{f'}\rt P')$ between two objects is a pair of $S$-homomorphisms $\psi_1:G\rt G'$ and $\psi_0:P\rt P'$ such that $\psi_0f=f'\psi_1$. It is clear that $\mon(\omega, \G)$ is a full additive subcategory of the monomorphism category of finitely generated $S$-modules. The full subcategory of $\mon(\omega, \G)$ consisting of those objects $(P\st{f}\rt Q)$ with $P,Q\in\cp(S)$, will be denoted by $\mon(\omega, \cp)$.
\end{dfn}

Recall that an exact category in the sense of Quillen is an additive category endowed with a class of kernel-cokernel pairs, called conflations, subject to certain axioms, see \cite[Definition 2.1]{buhler2010exact} and also \cite[Appendix A]{keller1990chain}. Let us give the precise definition.
{
\begin{dfn}(\cite[Definition 2.1]{buhler2010exact}) Let $\C$ be an additive category and let $\E$ be a class of kernel-cokernel pairs in $\C$. A kernel-cokernel pair $(i, p)$ in $\E$, which is also called an admissible pair, is a pair of composable morphisms $X'\st{i}\rt X\st{p}\rt X''$, where $i$ is a kernel of $p$ and $p$ is a cokernel of $i$. In this case, $i$ (resp. $p$) is called an admissible monic (resp. admissible epic).\\ The pair $(\C, \E)$ is called an exact category, if the following axioms are satisfied:
\begin{itemize} \item[$(E0)$] For each object $C\in\C$, the identity morphism $\id_C$ is an admissible monic. \item [$(E0^{op})$] For each object $C\in\C$, the identity morphism $\id_C$ is an admissible epic.\item[$(E1)$] The class of admissible moics is closed under composition.\item[$(E1^{op})$] The class of admissible epics is closed under composition.\item[$(E2)$] The push-out of an admissible monic along an arbitrary morphism exists and yields an admissible monic.\item[$(E2^{op})$] The pull-back of an admissible epic along an arbitrary morphism exists and yields an admissible epic.
\end{itemize}
Admissible pairs, admissible monics, and admissible epics are also called conflations, inflations and deflations, respectively, see \cite{keller1990chain} and \cite{quillen2006higher}.

For example, an abelian category is naturally an exact category whose conflations are the class of all short exact sequences. More generally, an extension-closed subcategory of an abelian category is an exact category in the same manner. This is a basic recognition principle of exact categories, for many examples arise in this way, see \cite[Lemma 10.20]{buhler2010exact}.
\end{dfn}

It is known that the morphism category of $S$-modules, $\mathsf{Mor}(S),$ is an abelian category, and so, it will be an exact category. As $\mon(\omega, \G)$ is a full subcategory of $\mathsf{Mor}(S)$,
we will observe that this exact structure is inherited by $\mon(\omega, \G)$. {Namely, we show that, considering  those short exact sequences in $\mathsf{Mor}(S)$ with terms in $\mon(\omega, \G)$ as conflations, $\mon(\omega, \G)$ is an exact category.}  In this direction, a natural way is to show that $\mon(\omega, \G)$ is an extension-closed subcategory of $\mathsf{Mor}(S)$. But, as Example \ref{non} reveals, this is not the case. So, we need to check all the axioms of exact categories directly. This will be done through a series of results.}

\begin{example}\label{non}
Assume that $(S, \n)$ is a one-dimensional commutative Gorenstein local ring and  $\omega\in\n$ is non-zero divisor. Assume that $M$ is an $S$-module that is not annihilated by $\omega$ and $\omega^2M=0$, e.g. $M=S/{(\omega^2)}$. Set $N:=0:_{M}\omega$, the submodule consisting of those objects in $M$ which are annihilated by $\omega$, and $K:=M/N$. Evidently, $K$ is also annihilated by $\omega$. Since $S$ is a one-dimensional Gorenstein ring, we may take short exact sequences of $S$-modules $0\rt G\st{f}\rt P\st{\pi}\rt N\rt 0$ and $0\rt G'\st{f'}\rt P'\st{\pi'}\rt K\rt 0$, where $G, G'\in\G(S)$ and $P, P'\in\cp(S)$, see for example \cite[Theorem 10.2.14]{enochs2011relative}. Thus $(G\st{f}\rt P)$ and $(G'\st{f'}\rt P')$ are objects of $\mon(\omega, \G)$. Applying the horseshoe lemma, gives us the short exact sequnce $0\lrt (G\st{f}\rt P)\lrt (G''\st{g}\rt P\oplus P')\lrt (G'\st{f'}\rt P')\lrt 0$ in $\mathsf{Mor}(S)$. Since $\cok g=M$ is not annihilated by $\omega$, the middle term does not belong to $\mon(\omega, \G)$. This indeed means that $\mon(\omega, \G)$ is not an extension-closed subcategory of $\mathsf{Mor}(S)$.
\end{example}

\begin{lemma}\label{admono}Let $(G\st{g}\rt P)\st{\varphi=(\varphi_1,\varphi_0)}\lrt(G'\st{g'}\rt P')$ and $(G'\st{g'}\rt P')\st{\theta=(\theta_1,\theta_0)}\lrt(G''\st{g''}\rt P'')$ be two morphisms in $\mon(\omega, \G)$ such that $\varphi$ and $\theta$ are injective and their cokernels are in $\mon(\omega, \G)$. Then $\cok\theta\varphi\in\mon(\omega, \G)$.
\end{lemma}
\begin{proof}Assume that $\cok\varphi=(L_1\st{l}\rt L_0)$, $\cok\theta=(E_1\st{e}\rt E_0)$ and $\cok\theta\varphi=(Z_1\st{z}\rt Z_0)$. By our assumption, $(L_1\st{l}\rt L_0)$ and $(E_1\st{e}\rt E_0)$ lie in $\mon(\omega, \G)$. We must show that the same is true for $(Z_1\st{z}\rt Z_0)$. {By our hypothesis, there are commutative diagrams of $S$-modules with exact rows \[\xymatrix{0\ar[r]& G \ar[r]^{\varphi_1} \ar[d]_{g} & G' \ar[r]^{\alpha_1} \ar[d]_{g'} & L_1 \ar[d]_{l}
\ar[r]&0\\0\ar[r]& P \ar[r]^{\varphi_0} & P' \ar[r]^{\alpha_0} & L_0\ar[r]&0, }\]
\[\xymatrix{0\ar[r]& G \ar[r]^{\theta_1\varphi_1} \ar[d]_{g} & G'' \ar[r]^{\alpha'_1} \ar[d]_{g''} & Z_1 \ar[d]_{z}
\ar[r]&0\\0\ar[r]& P \ar[r]^{\theta_0\varphi_0} & P'' \ar[r]^{\alpha'_0} & Z_0\ar[r]&0}\] and
\[\xymatrix{0\ar[r]& G' \ar[r]^{\theta_1} \ar[d]_{g'} & G'' \ar[r]^{\gamma_1} \ar[d]_{g''} & E_1 \ar[d]_{e}
\ar[r]&0\\0\ar[r]& P' \ar[r]^{\theta_0} & P'' \ar[r]^{\gamma_0} & E_0\ar[r]&0.}\]Consider the following commutative diagram with exact rows and columns:\[\xymatrix{G~\ar[r]^{\varphi_1}\ar@{=}[d]& G'\ar[r]^{\alpha_1}\ar[d]_{\theta_1}& L_1\ar[d]_{h_1}\\ G~\ar[r]^{\theta_1\varphi_1} & G''\ar[r]^{\alpha'_1}\ar[d]_{\gamma_1}& Z_1\ar[d]_{h_1'}\\ & E_{1}~\ar@{=}[r] & E_{1}.}\]Since $\theta_1$ (resp. $\gamma_1$) is a monomorphism (resp. an epimorphism), the same is true for $h_1$ (resp. $h'_1$).  In particular, we get the short exact sequence of $S$-modules, $0\rt L_1\st{h_1}\rt Z_1\st{h'_1}\rt E_1\rt 0.$
}
 {Also, replacing the composition morphism $G\st{\varphi_1}\rt G'\st{\theta_1}\rt G''$ with $P\st{\varphi_0}\rt P'\st{\theta_0}\rt P''$ in the latter diagram, gives us a short exact sequence of $S$-modules, $0\rt L_0\st{h_0}\rt Z_0\st{h_0'}\rt E_0\rt 0$. {Moreover, one obtains the equalities $h_0\alpha_0=\alpha'_0\theta_0$ and $h'_0\alpha'_0=\gamma_0$}. Now, by making use of the fact that $\alpha_1$ and $\alpha'_1$ are epimorphisms, one may get} the following commutative diagram with exact rows: \[\xymatrix{0\ar[r]& L_1 \ar[r]^{h_1} \ar[d]_{l} & Z_1 \ar[r]^{h_1'} \ar[d]_{z} & E_1 \ar[d]_{e}
\ar[r]&0\\0\ar[r]& L_0 \ar[r]^{h_0} & Z_0 \ar[r]^{h_0'} & E_0\ar[r]&0. }\]{To see this, consider the equalities: $zh_1\alpha_1=z\alpha'_1\theta_1=\alpha'_0g''\theta_1=\alpha'_0\theta_0g'=h_0\alpha_0g'=h_0l\alpha_1$. Now since $\alpha_1$ is an epimorphism, one infers that $zh_1=h_0l$. Similarly, one may obtain the equality $eh'_1=h'_0z$. As $(L_1\st{l}\rt L_0)$ and $(E_1\st{e}\rt E_0)$ are objects of $\mon(\omega, \G)$, we have that $L_1, E_1\in\G(S)$ and $L_0, E_0\in\cp(S)$. Now since} $\cp(S)$ and $\G(S)$ are closed under extensions, $Z_1\in\G(S)$ and $Z_0\in\cp(S)$. Moreover, as $l$ and $e$ are monomorphisms, the snake lemma yields that $z$ is also a monomorphism. So, it remains to show that $\cok z$ is annihilated by $\omega$. To this end, take the following commutative diagram: \[\xymatrix{0\ar[r]& G \ar[r]^{\theta_1\varphi_1} \ar[d]_{g} & G'' \ar[r]^{\alpha'_1} \ar[d]_{g''} & Z_1 \ar[d]_{z}\ar[r]&0
\\ 0\ar[r]&P \ar[r]^{\theta_0\varphi_0} & P'' \ar[r]^{\alpha'_0} & Z_0\ar[r]&0.}\] As $z$ is a monomorphism, another use of the snake lemma, gives us the short exact sequence of $S$-modules
$0\rt\cok g\rt\cok g''\rt\cok z\rt 0$. Now by our hypothesis, $\omega\cok g''=0$, we infer that $\cok z$ is also annihilated by $\omega$. So the proof is finished.
\end{proof}

\begin{remark}\label{epi} Suppose that $(G\st{g}\rt P)\st{\varphi=(\varphi_1,\varphi_0)}\lrt(G'\st{g'}\rt P')$ is an epimorphism in $\mon(\omega, \G)$. So, assuming $\Ker\varphi=(L_1\st{l}\rt L_0)$, we get the following commutative diagram with exact rows: \[\xymatrix{0\ar[r]& L_1 \ar[r] \ar[d]_{l} & G \ar[r]^{\varphi_1} \ar[d]_{g} & G' \ar[d]_{g'}\ar[r]&0
\\ 0\ar[r]&L_0 \ar[r] & P \ar[r]^{\varphi_0} & P'\ar[r]&0.}\] Since the class of (Gorenstein) projective modules is closed under kernels of epimorphisms, $L_1$ (resp. $ L_0$) will be Gorenstein projective (resp. projective). Furthermore, applying the snake lemma yields that $\cok l$ is annihilated by $\omega$, and so, $(L_1\st{l}\rt L_0)\in\mon(\omega, \G)$. Thus, the category $\mon(\omega, \G)$ is closed under kernels of epimorphisms.
\end{remark}

\begin{prop}\label{pushout}The push-out of any diagram {$$\begin{CD}(G\st{f}\rt P) @>{\varphi=(\varphi_1, \varphi_0)}>>(G'\st{g}\rt P')\\@V\theta=(\theta_1, \theta_0) VV \\
(G''\st{h}\rt P'')\end{CD}$$} in $\mon(\omega, \G)$, where $\varphi$ is an injective and $\cok\varphi\in\mon(\omega, \G)$, exists.
\end{prop}
\begin{proof}Assume that $\cok\varphi=(L_1\st{l}\rt L_0)$. Take the following push-out diagrams of $S$-modules:

$$\xymatrix{0\ar[r]&G\ar[r]^{\varphi_1} \ar[d]_{\theta_1} &\ G' \ar[d]^{\theta'_1}\ar[r]^{\beta_1}& L_1 \ar@{=}[d]\ar[r]& 0\\ 0\ar[r]& G'' \ar[r]^{\varphi'_1} & E_1 \ar[r]^{\beta'_1} & L_1\ar[r]& 0}$$and
$$\xymatrix{0\ar[r]&P\ar[r]^{\varphi_0} \ar[d]_{\theta_0} &\ P' \ar[d]^{\theta'_0}\ar[r]^{\beta_0}& L_0 \ar@{=}[d]\ar[r]&0\\ 0\ar[r]&P'' \ar[r]^{\varphi'_0} & E_0 \ar[r]^{\beta'_0} & L_0\ar[r] &0.}$$ By our assumption, $(L_1\st{l}\rt L_0)\in\mon(\omega, \G)$. This, in conjunction with the fact that $\G(S)$ and $\cp(S)$ are closed under extensions, would imply that $E_1$ is Gorenstein projective and $E_0$ is projective.
By the universal property of push-out diagrams, we may find a morphism $E_1\st{e}\rt E_0$, which makes the following diagram commutative.

{\footnotesize{\[\xymatrix{&G\ar[rr]^{\varphi_1}\ar[dd]_{\theta_1}\ar[ld]_{f}&&G'\ar[rr]\ar[dd]_{\theta'_1}\ar[dl]_{g}&&L_{1}\ar[dl]\ar@{=}[dd]\\ P\ar[rr]^{\varphi_0}\ar[dd]_{\theta_0}&& P'\ar[rr]\ar[dd]_{\theta'_0}&&L_{0}\ar@{=}[dd]\\&G''\ar[rr]^{\varphi'_1}\ar[dl]_{h}&&E_{1}\ar[rr]\ar@{-->}[dl]_{e}&&L_{1}\ar[dl]_{l}\\ P''\ar[rr]^{\varphi'_0}&&E_{0}\ar[rr]&&L_{0}}\]}}We claim that $(E_1\st{e}\rt E_0)\in\mon(\omega, \G)$. Since $h$ and $l$ are monomorphisms, the same is true for $e$. Now we show that $\cok e$ is annihilated by $\omega$. To do this, it suffices to show that, for a given object $x\in E_0$, {$\omega x\in{\imm}e$.
As $\cok\theta_0=\cok\theta'_0$, one may find objects $y\in P''$ and $z\in P'$ such that $x=\varphi'_0(y)+\theta'_0(z)$. So, $\omega x=\varphi'_0(\omega y)+\theta'_0(\omega z)$. Since $\omega\cok h=0$, we have that $\omega y\in {\imm}h$, and so, there is an object $t\in G''$ such that $h(t)=\omega y$. Thus one has the equalities
$\varphi'_0(\omega y)=\varphi'_0h(t)=e\varphi'_1(t)$. Similarly, as $\omega\cok g=0$, we infer that $\omega z\in{\imm}g$. Take an object $t'\in G'$ such that $g(t')=\omega z$. Hence we get the equalities $\theta'_0(\omega z)=\theta'_0g(t')=e\theta'_1(t')$. Consequently, $\omega x\in{\imm}e$, as claimed. In particular, we obtain the following commutative diagram with exact rows:$$\xymatrix{0\ar[r]&(G\st{f}\rt P)\ar[r]^{\varphi=(\varphi_1,\varphi_0)} \ar[d]_{\theta=(\theta_1,\theta_0)} &(G'\st{g}\rt P') \ar[d]^{\theta'=(\theta'_1,\theta'_0)}\ar[r]&( L_1\st{l}\rt L_0) \ar@{=}[d]\ar[r]& 0\\ 0\ar[r]& (G''\st{h}\rt P'') \ar[r]^{\varphi'=(\varphi'_1,\varphi'_0)} & (E_1\st{e}\rt E_0) \ar[r] & (L_1\st{l}\rt L_0)\ar[r]& 0,}$$ in $\mon(\omega, \G)$, which is indeed a push-out diagram. So the proof is finished.}
\end{proof}

The proof of the next result is dual to the proof of Proposition \ref{pushout} and so, we skip it.
\begin{prop}\label{pullback}The pull-back of any diagram {$$\begin{CD}&& (G''\st{h}\rt P'')\\&& @V\theta VV\\ (G\st{f}\rt P) @>{\varphi}>>(G'\st{g}\rt P'),\end{CD}$$} in $\mon(\omega, \G)$ with $\varphi$ epimorphism, exists.
\end{prop}
Now we are ready to state the result below.
\begin{prop}\label{sexact} $\mon(\omega, \G)$ is an exact category.
\end{prop}
\begin{proof}First one should note that the class of short exact sequences in $\mathsf{Mor}(S)$ with terms in $\mon(\omega, \G)$ is closed under isomorphisms. Moreover, axioms $(E0)$ and $(E0^{op})$ are obviously satisfied. The validity of axioms $(E1)$ comes from Lemma \ref{admono} and axiom $(E1^{op})$ is true, because of Remark \ref{epi}. Furthermore, axiom $(E2)$ is satisfied, thanks to Proposition \ref{pushout}. Finally, axiom $(E2^{op})$ follows from Proposition \ref{pullback}. So the proof is finished.
\end{proof}

\begin{cor}\label{ccor}$\mon(\omega, \cp)$ is an extension-closed exact subcategory of $\mon(\omega, \G)$.
\end{cor}
\begin{proof}Assume that $0\lrt (G'\st{g'}\rt P')\lrt (G\st{g}\rt P)\lrt (G''\st{g''}\rt P'')\lrt 0$ is a short exact sequence in $\mon(\omega, \G)$ such that $G',G''\in\cp(S)$. This gives us the short exact sequence of $S$-modules $0\rt G'\rt G\rt G''\rt 0$. Since $G'$ and $G''$ are projective, the same is true for $G$, and so, $(G\st{g}\rt P)\in\mon(\omega, \cp)$. This means that $\mon(\omega, \cp)$ is an extension-closed subcategory of $\mon(\omega, \G)$. Thus, the exactness of $\mon(\omega, \G)$ is inherited by $\mon(\omega, \cp)$, see \cite[Lemma 10.20]{buhler2010exact}. So the proof is completed.
\end{proof}
In what follows, we show that $\mon(\omega, \G)$ is a Frobenius category. An exact category $\A$ is called a Frobenius category, if it has enough projectives and injectives and the projectives coincide with the injectives. Here
we state a couple of elementary lemmas.

\begin{lemma}\label{1.1}For a given commutative diagram of $S$-modules\[\xymatrix{&\\ E:0 \ar[r] & M' \ar[r] \ar[d]_{\alpha} & M \ar[r] \ar[d]_{\beta} & M'' \ar[r] \ar[d]_{\gamma} & 0\\E': 0 \ar[r] & N' \ar[r] & N \ar[r] & N'' \ar[r] & 0}\] with exact rows, there exists a commutative diagram\[\xymatrix{&\\ E: 0 \ar[r] & M' \ar[r] \ar[d]_{\alpha} & M \ar[r] \ar[d]_{\beta_1} & M''\ar[r] \ar@{=}[d] & 0
\\ E'': 0 \ar[r] & N'\ar@{=}[d] \ar[r] & T \ar[r]\ar[d]_{\beta_2} & M'' \ar[r]\ar[d]_{\gamma} & 0\\ E': 0 \ar[r] & N' \ar[r] & N \ar[r] & N'' \ar[r] & 0,}\]such that $E''$ is exact and $\beta_2\beta_1=\beta$.
\end{lemma}
\begin{proof}See \cite[Lemma 1.1, page 163]{mitchell1965theory} and also \cite[Proposition 3.1]{buhler2010exact}.
\end{proof}

\begin{lemma}\label{split}Let $0\rt M'\st{f}\rt M\st{g}\rt M''\rt 0$ be a short exact sequence of $S$-modules such that $\omega M''=0$. Then the push-out of this sequence along the multiplicative morphism $M'\st{\omega}\rt M'$ is split.
\end{lemma}
\begin{proof}Consider the following commutative diagram with exact rows:
\[\xymatrix{&\\ 0 \ar[r] & M' \ar[r]^{f} \ar[d]_{\omega} & M \ar[r]^{g} \ar[d]_{\omega} & M'' \ar[r] \ar[d]_{\omega} & 0\\ 0 \ar[r] & M' \ar[r]^{f} & M \ar[r]^{g} & M'' \ar[r] & 0.}\]So, applying Lemma \ref{1.1} gives us the following commutative diagram with exact rows:
\[\xymatrix{&\\ 0 \ar[r] & M' \ar[r]^{f} \ar[d]_{\omega} & M \ar[r]^{g} \ar[d]_{h_1} & M''\ar[r] \ar@{=}[d] & 0
\\ 0 \ar[r] & M'\ar@{=}[d] \ar[r] & T \ar[r]\ar[d]_{h_2} & M'' \ar[r]\ar[d]_{\omega} & 0\\ \eta:0 \ar[r] & M' \ar[r]^{f} & M \ar[r]^{g} & M'' \ar[r] & 0,}\]such that $h_2h_1=\omega.\id_{M}$. Since $\omega M''=0$, the middle row, which is also the push-out of the top row along the morphism $M'\st{\omega}\rt M'$, will be split. So the proof is finished.
\end{proof}

\begin{lemma}\label{proj}Let $Q$ be a projective $S$-module. Then $(Q\st{\id}\rt Q)$ and $(Q\st{\omega}\rt Q)$ are projective and injective objects in $\mon(\omega, \G)$.
\end{lemma}
\begin{proof}We deal only with the case $(Q\st{\omega}\rt Q)$, because the other one is obtained easily. Let us first examine the projectivity of $(Q\st{\omega}\rt Q)$. To this end, take a short exact sequence $0\lrt (G'\st{g'}\rt P')\lrt (G\st{g}\rt P)\st{\varphi=(\varphi_1,\varphi_0)}\lrt (Q\st{\omega}\rt Q)\lrt 0$ in $\mon(\omega, \G)$. Now projectivity of $Q$ gives us a morphism $\psi_0:Q\rt P$ with $\varphi_0\psi_0=\id_Q$. Since $\cok g$ is annihilated by $\omega$, one may find a morphism $\psi_1:Q\rt G$ making the following diagram commutative
{\footnotesize\[\xymatrix{ & Q\ar[d]^{\psi_0\omega}\ar[dl]_{\psi_1} & \\ G\ar[r]^{g}~& P\ar[r]& \cok g.}\]}{So, $\psi=(\psi_1, \psi_0):(Q\st{\omega}\rt Q)\lrt (G\st{g}\rt P)$ is a morphism in $\mon(\omega, \G)$. Consider the equalities: $\omega\varphi_1\psi_1=(\omega\varphi_1)\psi_1=(\varphi_0g)\psi_1=\varphi_0(g\psi_1)=\varphi_0\psi_0\omega=\omega\varphi_0\psi_0$. Now since $\omega$ is non-zerodivisor, we have that $\varphi_1\psi_1=\varphi_0\psi_0.$ Consequently, $\varphi\psi=(\varphi_1\psi_1, \varphi_0\psi_0)=(\id_Q, \id_Q)=\id_{(Q\st{\omega}\rt Q)}$, and so $(Q\st{\omega}\rt Q)$ is a projective object of $\mon(\omega, \G)$.} Next, we show that $(Q\st{\omega}\rt Q)$ is an injective object of $\mon(\omega, \G)$.
Assume that $0\lrt (Q\st{\omega}\rt Q)\st{\varphi}\lrt (G\st{g}\rt P)\lrt (G'\st{g'}\rt P')\lrt 0$ is a short exact sequence in $\mon(\omega, \G)$. Since $Q$ is an injective object in $\G(S)$, there is a morphism $\psi_1:G\rt Q$ such that $\psi_1\varphi_1=\id_Q$.
Consider the following push-out diagram:
\[\xymatrix{&\\ \gamma:0 \ar[r] & G \ar[r]^{g} \ar[d]_{\omega\psi_1} & P \ar[r] \ar[d] & \cok g \ar[r] \ar@{=}[d] & 0
\\ (\omega\psi_1)\gamma: 0 \ar[r] & Q \ar[r] & T \ar[r] & \cok g \ar[r] & 0. &}\]
By our hypothesis, $\omega\cok g=0$, and so, applying Lemma \ref{split} yields that $(\omega \psi_1)\gamma$ is split. Hence, there is a morphism $\psi_0:P\lrt Q$ such that $\psi_0g=\omega \psi_1$. Now it can be easily checked that $\psi\varphi=\id_{(Q\st{\omega}\rt Q)}$. Namely, $(Q\st{\omega}\rt Q)$ is an injective object of $\mon(\omega, \G)$, and so, the proof is completed.
\end{proof}

\begin{lemma}\label{si}For a given object $(G\st{g}\rt P)\in\mon(\omega, \G)$, there is a unique $S$-homomorphism $g_{\si}:P\rt G$ such that $\omega\cok g_{\si}=0$, $g_{\si}g=\omega.\id_{G}$ and $gg_{\si}=\omega.\id_{P}$.
\end{lemma}
\begin{proof}Take the following push-out diagram: \[\xymatrix{&\\ 0 \ar[r] & G \ar[r]^{g} \ar[d]_{\omega} & P \ar[r] \ar[d] & \cok g \ar[r] \ar@{=}[d] & 0\\ 0 \ar[r] & G \ar[r] & L \ar[r] & \cok g \ar[r] & 0. &}\]
As $\omega\cok g=0$, by Lemma \ref{split} the lower row is split, and so, there is a morphism $g_{\si}:P\rt G$ such that $g_{\si}g=\omega.\id_{G}$. Another use of the fact that $\omega\cok g=0$, leads us to infer that $\omega P\subseteq g(G)$. This fact besides the equality $g_{\si}g=\omega.\id_{G}$ would imply that $gg_{\si}=\omega.\id_{P}$. {It should be noted that, if there is another morphism $f: P\rt G$ satisfying the mentioned conditions, then we will have $gg_{\si}=gf$, and then, $g$ being a monomorphism ensures the validity of the equality $g_{\si}=f$.} Finally, we show that $\omega$ annihilates $\cok g_{\si}$. To see this, consider the short exact sequence of $S$-modules; $0\rt P\st{g_{\si}}\rt G\st{\pi}\rt\cok g_{\si}\rt 0$. For a given object $y\in\cok g_{\si}$, take $x\in G$ such that $\pi(x)=y$. Hence, $\omega y=\omega\pi(x)=\pi(\omega x)=\pi(g_{\si}g(x))=0$, meaning that $\omega\cok g_{\si}=0$. Thus the proof is completed.
\end{proof}

\begin{prop}\label{eno}The category $\mon(\omega, \G)$ has enough projective and injective objects.
\end{prop}
\begin{proof}Consider an arbitrary object $(G\st{g}\rt P)\in\mon(\omega, \G)$. In view of Lemma \ref{si}, there is a unique $S$-homomorphism $g_{\si}:P\rt G$ such that $gg_{\si}=\omega.\id_{P}$ and its cokernel is annihilated by $\omega$. So, taking an epimorphism $\pi:Q\rt G$ with $Q\in\cp(S)$, one may get an epimorphism $(Q\oplus P\st{\id\oplus\omega}\lrt Q\oplus P)\st{\varphi}\lrt (G\st{g}\rt P)$ in $\mon(\omega, \G)$, where $\varphi=(\varphi_1~~\varphi_0)$ with $\varphi_1=[\pi~~g_{\si}]$ and $\varphi_0=[g\pi~~\id]$. Now Remark \ref{epi} combined with Lemma \ref{proj} yields the desired result for the projective case. Now we focus on the injective case. Since $G\in\G(S)$, there exists a short exact sequence of $S$-modules $0\rt G\st{h}\rt Q\st{h'}\rt G'\rt 0$, where $Q$ is projective and $G'$ is Gorenstein projective. So one may obtain a monomorphism $(G\st{g}\rt P)\st{\varphi}\lrt(Q\oplus P\st{\omega\oplus\id}\lrt Q\oplus P)$ in $\mon(\omega, \G)$, where $\varphi_1=[h~~g]^t$ and $\varphi_0=[hg_{\si}~~\id]^t$. According to Lemma \ref{proj}, $(Q\oplus P\st{\omega\oplus\id}\lrt Q\oplus P)$ is an injective object of $\mon(\omega, \G)$. Hence, it remains to show that $\cok\varphi\in\mon(\omega, \G)$. Since $P$ is projective and $G'$ is Gorenstein projective, one may find a morphism $\alpha: Q\rt P$ with $\alpha h=g$. Thus we will have the equality $(\omega-hg_{\si}\alpha)h=0$, and this in turn, guarantees the existence of a morphism $\beta: G'\rt Q$ such that $\beta h'=\omega-hg_{\si}\alpha$. Now by letting $l=[\beta~~-hg_{\si}]$, one may deduce that $(G'\oplus P\st{l}\rt Q)$ is an object of $\mon(\omega, \G)$. In fact, considering the morphism $(Q\st{l_0}\rt G'\oplus P)$, where $l_0=[h'~~-\alpha]^t$, we have $l_0l=\omega\oplus\omega$, and in particular, $l$ is a monomorphism. Now it is easy to see that $$0\lrt(G\st{g}\rt P)\st{\varphi}\lrt(Q\oplus P\st{\omega\oplus\id}\lrt Q\oplus P)\st{\psi}\lrt (G'\oplus P\st{l}\rt Q)\lrt 0,$$with $\psi_0=[\id~~-hg_{\si}]$ and $\psi_1=\tiny {\left[\begin{array}{ll} h' & 0 \\ -\alpha & \id \end{array} \right]}$, is a short exact sequence of monomorphisms. Since $\cok(\omega\oplus\id)$ is annihilated by $\omega$, applying the snake lemma yields that the same is true for $\cok l$. Consequently, the latter sequence is a short exact sequence in $\mon(\omega, \G)$, and in particular, $\cok\varphi\in\mon(\omega, \G)$. Hence the proof is completed.
\end{proof}

\begin{lemma}\label{inj}
Each projective and injective object in $\mon(\omega, \G)$ is equal to direct summands of finite direct sums of objects of the form $(Q\st{\id}\rt Q)\oplus(P\st{\omega}\rt P)$ for some projective $S$-modules $P, Q$. In particular, an object is injective if and only if it is projective.
\end{lemma}
\begin{proof}Let us deal only with the projective case, because the other one is obtained dually. Take an arbitrary projective object $(G\st{g}\rt P)\in\mon(\omega, G)$. According to the proof of Proposition \ref{eno}, there is an epimorphism $(Q\oplus P\st{\id\oplus\omega}\lrt Q\oplus P)\st{\varphi}\lrt (G\st{g}\rt P)$ in $\mon(\omega, \G)$, with $Q$ projective. Since $(G\st{g}\rt P)$ is projective, $\varphi$ is a split epimorphism, and so, $(G\st{g}\rt P)$ will be a direct summand of $(Q\st{\id}\rt Q)\oplus (P\st{\omega}\rt P)$, giving the desired result. Moreover, the second assertion follows from the first assertion together with Lemma \ref{proj}. So the proof is finished.
\end{proof}

\begin{theorem}\label{frob}$\mon(\omega, \G)$ is a Frobenius category.
\end{theorem}
\begin{proof}According to Proposition \ref{sexact}, $\mon(\omega, \G)$ is an exact category. Moreover, this category has enough projective and injective objects by Proposition \ref{eno}. Finally, Lemma \ref{inj} reveals that projective objects are the same as injective objects in $\mon(\omega, \G)$. So the proof is finished.
\end{proof}

Assume that $\A$ is a Frobenius category and $\mathcal{B}$ an extension-closed exact subcategory of $\A$. Then $\mathcal{B}$ is said to be an {\em admissible subcategory}, if each object $B\in\mathcal{B}$ fits into conflations  $B\rt Q\rt B'$ and $B''\rt P\rt B$ in $\mathcal{B}$ such that $P, Q$ are projective objects of $\A$. One should note that, in this case, $\mathcal{B}$ is also a Frobenius category, and in particular, an object $B\in\mathcal{B}$ is projective if and only if it is projective as an object of $\A$,  see for example \cite[page 46]{chen2012three}.

\begin{cor}\label{s1}$\mon(\omega, \cp)$ is a Frobenius subcategory of $\mon(\omega, \G)$. In particular, they have the same projective objects.
\end{cor}
\begin{proof}
At first, one should note that by Corollary \ref{ccor}, $\mon(\omega, \cp)$ is an extension-closed exact subcategory of $\mon(\omega, \G).$ Take an arbitrary object $(P\st{f}\rt Q)\in\mon(\omega, \cp)$. Since $\mon(\omega, \G)$ is a Frobenius category, one may get short exact sequences $0\lrt(G\st{g}\rt P_1)\lrt(Q_1\st{h}\rt Q_1)\lrt (P\st{f}\rt Q)\lrt 0$ and $0\lrt (P\st{f}\rt Q)\lrt (Q'_1\st{h'}\rt Q'_1)\lrt (G'\st{g'}\rt P'_1)\lrt 0$ in $\mon(\omega, \G)$, such that the middle terms are projective objects of $\mon(\omega, \G)$. Now since $\cp(S)$ is closed under kernels of epimorphisms, $G$ will be a projective $S$-module. Also, by making use of the fact that any Gorenstein projective $S$-module of finite projective dimension is projective, we infer that $G'$ is projective. Consequently, these short exact sequences lie in the category $\mon(\omega, \cp)$, and in particular, $\mon(\omega, \cp)$ is an admissible subcategory of $\mon(\omega, \G)$. Thus, as mentioned above, $\mon(\omega, \cp)$ is a Frobenius category and its projective objects are the same as $\mon(\omega, \G)$. So the proof is finished.
\end{proof}

The significance of Frobenius categories lies in their natural connection to triangulated categories. Namely, if $\mathcal{A}$ is a Frobenius category, then the stable category $\underline{\mathcal{A}}$ is a triangulated category such that its shift functor is given by the inverse of the syzygy functor on $\underline{\mathcal{A}}$, see \cite{happel1988triangulated, keller1990chain}. These types of triangulated categories are referred to as “algebraic” in the context of Keller’s definition \cite{keller2006differential}.

\begin{remark}In \cite{bahlekeh2023homotopy}, a homotopy category $\HT\mon(\omega, \cp)$ of $\mon(\omega, \cp)$ has been introduced and studied. The objects of $\HT\mon(\omega, \cp)$ are the same as $\mon(\omega, \cp)$ and its morphism sets are morphism sets in $\mon(\omega, \cp)$ modulo null-homotopic.  A morphism  $\psi=(\psi_1, \psi_0):(P\st{f}\rt Q)\lrt (P'\st{f'}\rt Q')$ in $\mon(\omega, \cp)$ is said to be  null-homotopic, if there are  $S$-homomorphisms $s_1:P\rt Q'$ and $s_0:Q\rt P'$ such that $f'\psi_1-f's_0f=\omega. s_1$, or equivalently, $\psi_0f-f's_0f=\omega. s_1$, see \cite[Definition 2.2]{bahlekeh2023homotopy}. As shown in \cite[Proposition 2.12]{bahlekeh2023homotopy}, $\HT\mon(\omega, \cp)$ admits a natural structure of triangulated category. Moreover, in view of \cite[Proposition 3.2]{bahlekeh2023homotopy}, a morphism $\psi=(\psi_1, \psi_0)$ in $\mon(\omega, \cp)$ is null-homotopic if and only if it factors through a projective object of $\mon(\omega, \cp)$. This in particular reveals that the triangulated structures of $\HT\mon(\omega, \cp)$ and $\umon(\omega, \cp)$ coincide.
\end{remark}

\begin{theorem}The inclusion functor $i:\mon(\omega, \cp)\rt\mon(\omega, \G)$ induces a fully faithful triange functor $\underline{i}:\umon(\omega, \cp)\rt\umon(\omega, \G)$.
\end{theorem}
\begin{proof}According to Corollary \ref{s1}, the Frobenius categories $\mon(\omega, \cp)$ and $\mon(\omega, \G)$ have the same projective objects. So applying \cite[ Lemma 2.8, page 23]{happel1988triangulated} yields that the inclusion functor $i:\mon(\omega, \cp)\lrt\mon(\omega, \G)$, induces a fully faithful triangle functor $\underline{i}:\umon(\omega, \cp)\lrt\umon(\omega, \G)$, as needed.
\end{proof}

\section{Stable category of monomorphisms and comparison with categories of singularities and D-branes of type B}

This section contains two important results. The first one indicates that the stable category $\umon(\omega, \cp)$ is triangle equivalent to the category of D-branes of type $B$, $\DB(\omega)$,
while the second one says that there are fully faithful triangle functors from the stable categories
$\umon(\omega, \G)$ and $\umon(\omega, \cp)$ to the singularity category $\ds(R)$ of $R$. Moreover, the density of these functors guarantees regularity of $S$ and Gorenstieness of $R$ (and so, $S$).
We begin by recalling the construction of the category D-branes of type B.

\begin{s}\label{matrix}We remind that the category of pairs $\Pair(\omega)$ is the category whose objects are ordered pairs ${\xymatrix{(P_1 \ar@<0.6ex>[r]^{\rho_1}& P_0\ar@<0.6ex>[l]^{\rho_0})}}$ in which $P_1$ and $P_0$ are projective $S$-modules and the compositions $\rho_0\rho_1$ and $\rho_1\rho_0$ are the multiplications by $\omega$, and a morphism $\Psi=(\psi_1, \psi_0) :{\xymatrix{(P_1 \ar@<0.6ex>[r]^{\rho_1}& P_0\ar@<0.6ex>[l]^{\rho_0})}}\rt {\xymatrix{(Q_1 \ar@<0.6ex>[r]^{q_1}& Q_0\ar@<0.6ex>[l]^{q_0})}}$ in $\Pair(\omega)$ is a pair of $S$-homomorphisms $\psi_1:P_1\rt Q_1$ and $\psi_0:P_0\rt Q_0$ such that $\psi_1\rho_0=q_0\psi_0$ and $q_1\psi_1=\psi_0\rho_1$. Moreover, the category $\DB(\omega)$, which is the same as the homotopy category of matrix factorizations, is the category whose objects are the same as objects of $\Pair(\omega)$, and its morphism sets are morphism sets in $\Pair(\omega)$ modulo null-homotopic. Recall that a morphism $\Psi=(\psi_1, \psi_0) :{\xymatrix{(P_1 \ar@<0.6ex>[r]^{\rho_1}& P_0\ar@<0.6ex>[l]^{\rho_0})}}\lra {\xymatrix{(Q_1 \ar@<0.6ex>[r]^{q_1}& Q_0\ar@<0.6ex>[l]^{q_0})}}$
is said to be {\em null-homotopic}, if there are $S$-homomorphisms $s_0:P_0\rt Q_1$ and $s_1:P_1\rt Q_0$ such that $\psi_0=q_1s_0+s_1\rho_0$ and $\psi_1=s_0\rho_1+q_0s_1$. The category $\DB(\omega)$ naturally gets a triangulated structure. Indeed, the exact triangles are those isomorphic to the standard triangles using mapping cones, see \cite{ bergh2015complete, orlov2003triangulated}.

It is known that $\Pair(\omega)$ is a Frobenius category, and in particular, its injective objects are those homotopic to zero pairs. This leads that the triangulated category $\underline{\Pair}(\omega)$ is nothing more than $\DB(\omega)$, see \cite{orlov2003triangulated}.

We should stress again that, the categories $\Pair(\omega)$ and $\DB(\omega)$ are the same as the category of matrix factorizations and the homotopy category of matrix factorizations, respectively. Matrix factorizations were introduced by Eisenbud in his 1980 paper \cite{eisenbud1980homological}, as a means of compactly describing the minimal free resolutions of maximal Cohen-Macaulay modules that have no free direct summands over a local hypersurface ring, see \cite{yoshino1990cohen}.
\end{s}

\begin{theorem}\label{sg}There is a triangle equivalence functor $\underline{F}:\umon(\omega, \cp)\lrt\underline{\Pair}(\omega)$.
\end{theorem}
\begin{proof}
First we define a functor $F:\mon(\omega, \cp)\lrt{\Pair}(\omega)$, as follows.
Take an arbitrary object $(P_1\st{\rho_1}\rt P_0)\in\mon(\omega, \cp)$. According to Lemma \ref{si}, there is a unique morphism $\rho_0:P_0\rt P_1$ such that $\rho_1\rho_0=\omega.\id_{P_0}$ and $\rho_0\rho_1=\omega.\id_{P_1}$. Namely, ${\xymatrix{(P_1 \ar@<0.6ex>[r]^{\rho_1}& P_0\ar@<0.6ex>[l]^{\rho_0})}}\in \Pair (\omega)$.
Now we set $F(P_1\st{\rho_1}\rt P_0):={\xymatrix{(P_1 \ar@<0.6ex>[r]^{\rho_1}& P_0\ar@<0.6ex>[l]^{\rho_0})}}$. Next assume that $\psi=(\psi_1, \psi_0):(P_1\st{\rho_1}\rt P_0)\lrt(P'_1\st{\rho'_1}\rt P'_0)$ is a morphism in $\mon(\omega, \cp)$. We would like to show that $\Psi=(\psi_1, \psi_0):{\xymatrix{(P_1 \ar@<0.6ex>[r]^{\rho_1}& P_0\ar@<0.6ex>[l]^{\rho_0})}} \lrt {\xymatrix{(P_1' \ar@<0.6ex>[r]^{\rho'_1}& P'_0\ar@<0.6ex>[l]^{\rho'_0})}}$ is a morphism in $\Pair(\omega)$. Namely, we show that the diagram
\[\xymatrix{ P_1 \ar[r]^{\rho_1} \ar[d]_{\psi_1} & P_0 \ar[r]^{\rho_0} \ar[d]_{\psi_0} & P_1 \ar[d]_{\psi_1}
\\ P'_1 \ar[r]^{\rho'_1} & P'_0 \ar[r]^{\rho'_0} & P'_1,}\] commutes. It is only needed to examine the commutativity of the right square, because by the hypothesis the left one is commutative. This indeed follows from the equalities: $\rho'_1(\rho'_0\psi_0-\psi_1\rho_0)=\rho'_1\rho'_0\psi_0-\rho'_1\psi_1\rho_0=\omega\psi_0-\psi_0\omega=0$, and the fact that $\rho'_1$ is a monomorphism. So, we define  $F(\psi):=\Psi=(\psi_1, \psi_0)$. It is evident that $F$ is an additive equivalence functor. We claim that $F$ is an exact functor. To do this, take a short exact sequence $$0\lrt(P'_1\st{\rho'_1}\rt P'_0)\st{\varphi=(\varphi_1,\varphi)}\lrt (P_1\st{\rho_1}\rt P_0)\st{\psi=(\psi_1,\psi_0)}\lrt (P''_1\st{\rho''_1}\rt P''_0)\lrt 0,$$ in $\mon(\omega, \cp)$. As observed just above, there is a commutative diagram

$$\xymatrix{0\ar[r]& P'_1\ar[r]^{\varphi_1} \ar[d]_{\rho'_1} & P_1 \ar[d]_{\rho_1}\ar[r]^{\psi_1}& P''_1 \ar[d]_{\rho''_1}\ar[r]& 0\\ 0\ar[r]& P'_0 \ar[r]^{\varphi_0}\ar[d]_{\rho'_0} & P_0 \ar[d]_{\rho_0}\ar[r]^{\psi_0} & P''_0\ar[d]^{\rho''_0}\ar[r]& 0 \\ 0\ar[r] & P'_1 \ar[r]^{\varphi_1} & P_1\ar[r]^{\psi_1}& P''_1\ar[r]& 0.}$$One should note that, by our assumption, rows are exact. This means that
$$0\lrt {\xymatrix{(P'_1 \ar@<0.6ex>[r]^{\rho'_1}& P'_0\ar@<0.6ex>[l]^{\rho'_0})}}\st{\Phi=(\varphi_1,\varphi_0)}\lra {\xymatrix{(P_1 \ar@<0.6ex>[r]^{\rho_1}& P_0)\ar@<0.6ex>[l]^{\rho_0}}}\st{\Psi=(\psi_1,\psi_0)}\lra {\xymatrix{(P''_1 \ar@<0.6ex>[r]^{\rho''_1}& P''_0)\ar@<0.6ex>[l]^{\rho''_0}}}\lra 0,$$ is a short exact sequence in $\Pair(\omega)$, and so, $F$ is an exact functor, as claimed. Moreover, it is obvious that the functor $F$ carries projective objects in $\mon(\omega, \cp)$ to projective objects in $\Pair(\omega)$. Indeed, it follows from our definition that $F(P\st{\id}\rt P)=\xymatrix{(P \ar@<0.6ex>[r]^{\id}& P)\ar@<0.6ex>[l]^{\omega}}$ and $F(P\st{\omega}\rt P)=\xymatrix{(P \ar@<0.6ex>[r]^{\omega}& P)\ar@<0.6ex>[l]^{\id}}$, which are projective objects of $\Pair(\omega)$, see \cite{orlov2003triangulated} and also \cite[Theorem 2.4]{bahlekehgpair}. Now Lemma \ref{inj} together with \cite[Theorem 2.4]{bahlekehgpair}, gives the desired result. Hence there is an induced functor $\underline{F}:\umon(\omega, \cp)\lrt\underline{\Pair}(\omega)$, see \cite[ 2.8, page 22]{happel1988triangulated}. Evidently, the induced functor $\underline{F}$ is also an equivalence. So it remains to show that $\underline{F}$ is a triangle functor. According to \cite[ Lemma 2.8, page 23]{happel1988triangulated}, it suffices to show that
$\underline{F}\syz^{-1}_{\mon}=\syz^{-1}_{\Pair}\underline{F}$. Take an arbitrary object $(P\st{f}\rt Q)\in\umon(\omega, \cp)$. It is easily seen that $0\lrt(P\st{f}\rt Q)\st{i=(i_1, i_0)}\lrt (P\oplus Q\st{\omega\oplus\id}\rt P\oplus Q)\st{\pi=(\pi_1, \pi_0)}\lrt (Q\st{-f_{\si}}\rt P)\lrt 0$ is a short exact sequence in $\mon(\omega, \cp)$, where $i_1=[\id~~f]^t$, $i_0=[f_{\si}~~\id]^t$, $\pi_1=[-f~~\id]$ and $\pi_0=[\id~~-f_{\si}]$. This yields that $\syz^{-1}_{\mon}(P\st{f}\rt Q)=(Q\st{-f_{\si}}\rt P)$, because the middle term is projective. Thus $\underline{F}(\syz^{-1}_{\mon}(P\st{f}\rt Q))=\xymatrix{(Q \ar@<0.6ex>[r]^{-f_{\si}}& P\ar@<0.6ex>[l]^{-f})}$, thanks to the fact that ${(f_{\si})}_{\si}=f$. Similarly, one may see that $0\lrt \xymatrix{(P \ar@<0.6ex>[r]^{f}& Q\ar@<0.6ex>[l]^{f_{\si}})}\lrt \xymatrix{(P\oplus Q \ar@<0.6ex>[r]^{\omega\oplus\id}& P\oplus Q\ar@<0.6ex>[l]^{\id\oplus\omega})}\lrt \xymatrix{(Q \ar@<0.6ex>[r]^{-f_{\si}}& P\ar@<0.6ex>[l]^{-f})}\lrt 0$ is a short exact sequence in $\Pair(\omega)$. Again, as the middle term is a projective object of $\Pair(\omega)$, the equality $\syz^{-1}_{\Pair}\xymatrix{(P \ar@<0.6ex>[r]^{f}& Q\ar@<0.6ex>[l]^{f_{\si}})}=\xymatrix{(Q \ar@<0.6ex>[r]^{-f_{\si}}& P\ar@<0.6ex>[l]^{-f})}$ holds. Now, since $\underline{F}(P\st{f}\rt Q)=\xymatrix{(P \ar@<0.6ex>[r]^{f}& Q\ar@<0.6ex>[l]^{f_{\si}})},$ we get that $\syz^{-1}_{\Pair}\underline{F}(P\st{f}\rt Q)=\xymatrix{(Q \ar@<0.6ex>[r]^{-f_{\si}}& P\ar@<0.6ex>[l]^{-f})}$. {Next, take a morphism $\psi=(\psi_1, \psi_0):(P_1\st{f}\rt P_0)\lrt(P'_1\st{f'}\rt P'_0)$ in $\umon(\omega, \cp)$. So, one may get the commutative diagram{\footnotesize{\[\xymatrix{&P_1\ar[rr]^{[\id~~f]^t}\ar[dd]_{\psi_1}\ar[ld]_{f}&&P_1\oplus P_0\ar[rr]^{[-f~~\id]}\ar[dd]\ar[dl]_{\omega\oplus\id}&&P_0\ar[dl]_{-f_{\si}}\ar[dd]_{\psi_0}\\ P_0\ar[rr]^{[f_{\si}~~\id]^t}\ar[dd]_{\psi_0}&& P_1\oplus P_0\ar[rr]^{[\id~~-f_{\si}]}\ar[dd]&&P_1\ar[dd]_{\psi_1}\\&P'_1\ar[rr]^{[\id~~f']^t}\ar[dl]_{f'}&&P'_1\oplus P'_0\ar[rr]^{[-f'~~\id]}\ar[dl]_{\omega\oplus\id}&&P'_0\ar[dl]_{-f'_{\si}}\\ P'_0\ar[rr]^{[f'_{\si}~~\id]^t}&&P'_1\oplus P'_0\ar[rr]^{[\id~~-f'_{\si}]}&&P'_1,}\]}}where the morphisms in the middle columns, $P_1\oplus P_0\lrt P'_1\oplus P'_0$, are $\psi_1\oplus\psi_0$. This diagram reveals that $\syz^{-1}_{\mon}(\psi)=\psi_{\si}:=(\psi_0, \psi_1)$ in $\umon(\omega, \cp)$. Furthermore, for a given morphism $\Psi=(\psi_1, \psi_0)$ in $\Pair(\omega)$, one may construct a similar diagram, indicating that $\syz^{-1}_{\Pair}(\Psi)=\Psi_{\si}=(\psi_0, \psi_1)$ in $\underline{\Pair}(\omega)$. So, according to our notation, $\underline{F}(\psi_{\si})=\Psi_{\si}$.
Now suppose that $\psi=(\psi_1, \psi_0)$ is a morphism in $\mon(\omega, \cp)$. As already observed, we have the equalities: $\syz^{-1}_{\Pair}\underline{F}(\psi)=\syz^{-1}_{\Pair}(\Psi)=\Psi_{\si}$. Moreover, $\underline{F}\syz_{\mon}^{-1}(\psi)=\underline{F}(\psi_{\si})=\Psi_{\si}$.} So the proof is finished.
\end{proof}

As we have already mentioned above, the category $\DB(\omega)$ is triangle equivalent to the stable category $\underline{\Pair}(\omega)$. This, together with Theorem \ref{sg}, yields the following result.
\begin{cor}\label{sg1}There is a triangle equivalence functor $F':\umon(\omega, \cp)\lrt\DB(\omega)$.
\end{cor}

{The result below has been proved in \cite[Proposition 3.6]{bahlekehgpair}, under the assumption that $R$ is Gorenstein.
\begin{prop}\label{dd}There is a triangle equivalence functor $\underline{T}_1:\umon(\omega, \G)\lrt\underline{\G}(R)$, sending each object $(G\st{f}\rt P)$ to $\cok f$.
\end{prop}
\begin{proof}
Take an arbitrary object $(G\st{f}\rt P)\in\mon(\omega, \G)$. We show that $\cok f\in\G(R)$. Set, for the simplicity, $M:=\cok f$. As $\gpd_SM\leq 1$, $\Ext^i_S(M, S)=0$ for all $i\geq 2$, and so, applying \cite[Lemma 2(i), page 140]{matsumura1989commutative} yields that $\Ext_R^i(M, R)=0$ for all $i\geq 1$. Moreover, applying the functor $M\otimes_S-$ to the short exact sequence of $S$-modules $0\rt S\st{\omega}\rt S\rt R\rt 0$ and using the fact that $\omega M=0$, enable us to infer that $\Tor_1^S(M, R)\cong M$. This, in turn, gives us the exact sequence of $R$-modules $0\rt M\rt\bar{G}\rt\bar{P}\rt M\rt 0$, where $\bar{(-)}=S/{(\omega)}\otimes_S-.$ Letting $L:=\Ker(\bar{P}\rt M)$, one may easily see that $\Ext^i_R(L, R)=0$, for all $i\geq 1$. Consequently, the sequence of $R$-modules $0\rt M^*\rt{\bar{P}}^*\rt{\bar{G}}^*\rt M^*\rt 0$ is exact, where $(-)^*=\Hom_R(-, R)$. Thus we will get the following commutative diagram of $R$-modules:
\[\xymatrix{0\ar[r]& M \ar[r] \ar[d]_{\varphi_1} & \bar{G} \ar[r] \ar[d]_{\varphi_2} & \bar{P} \ar[d]_{\varphi_3}\ar[r]& M\ar[r]\ar[d]_{\varphi_1}& 0\\ 0\ar[r] & M^{**} \ar[r] & {\bar{G}}^{**} \ar[r] & {\bar{P}}^{**}\ar[r]^{\psi} & M^{**}.}\]Since the morphisms $\varphi_2$ and $\varphi_3$ are isomorphism, one may observe that the same is true for $\varphi_1$. Thus, $\psi$ will be an epimorphism, and in particular, two rows will be isomorphic. Now since $\Ext^i_R({\bar{G}}^*, R)=0$ for all $i\geq 1$, we deduce that $\Ext^i_R(M^*, R)=0$ for all $i\geq 1$, and so, $M\in\G(R)$. {Now we put $T_1(G\st{g}\rt P):=\cok g$. Moreover, for a given morphism $\psi=(\psi_1, \psi_0): (G\st{g}\rt P)\lrt (G'\st{g'}\rt P')$ in $\mon(\omega, \G)$, there is a unique induced morphism $\cok g\rt\cok g'$ in $\G(R)$, and we define $T_1(\psi):=(\cok g\rt\cok g')$. It is clear that $T_1:\mon(\omega, \G)\lrt\G(R)$ is a functor. Assume that $0\lrt (G'\st{g'}\rt P')\lrt (G\st{g}\rt P)\lrt (G''\st{g''}\rt P'')\lrt 0$ is a short exact sequence in $\mon(\omega, \G)$. Applying the snake lemma gives us the short exact sequence $0\rt\cok g'\rt\cok g\rt\cok g''\rt 0$ in $\G(R)$. This means that $T_1$ is an exact functor. As it was shown in the proof of \cite[Proposition 3.6]{bahlekehgpair}, $T_1$ sends projective objects to projective object, and so, it induces the functor $\underline{T}_1:\umon(\omega, \G)\lrt\underline{\G}(R)$. In particular, this functor is an equivalence of categories, see \cite[Proposition 3.6]{bahlekehgpair}.} Now we show that $\underline{T}_1$ is a triangle functor. To do this, take an arbitrary object $(G\st{g}\rt P)\in\mon(\omega, \G)$. Since $G$ is Gorenstein projective, there is a short exact sequence of $S$-modules, $0\rt G\st{h}\rt Q\st{h'}\rt G_1\rt 0$, where $Q$ is projective and $G_1$ is Gorenstein projective. As we have observed in the proof of Proposition \ref{eno}, there is a short exact sequence $$0\lrt (G\st{g}\rt P)\lrt(Q\oplus P\st{\omega\oplus\id}\lrt Q\oplus P)\lrt (G_1\oplus P\st{l}\rt Q)\lrt 0,$$in $\mon(\omega, \G)$. Since the middle term is projective, $\syz^{-1}_{\mon}(G\st{g}\rt P)=(G_1\oplus P\st{l}\rt Q)$, and so, $\underline{T}_1(\syz^{-1}_{\mon}(G\st{g}\rt P))=\cok l$. On the other hand, applying the snake lemma to the latter short exact sequence, gives us the short exact sequence, $0\rt\cok g\rt\cok(\omega\oplus\id)\rt\cok l\rt 0$ in $\G(R)$. This, in particular, reveals that $\syz^{-1}_{\G}\cok g=\cok l$, because $\cok(\omega\oplus\id)$ is a projective $R$-module. Namely, $\syz^{-1}_{\G}\underline{T}_1(G\st{g}\rt P)=\cok l$. Hence, $\underline{T}_1(\syz^{-1}_{\mon}(G\st{g}\rt P))=\syz^{-1}_{\G}\underline{T}_1(G\st{g}\rt P)$ in $\underline{\G}(R)$. {Next for a given morphism $\psi=(\psi_1, \psi_0):(G\st{g}\rt P)\lrt(G'\st{g'}\rt P')$ in $\umon(\omega, \G)$, one may easily get the following commutative diagram:
\[\xymatrix{0\ar[r]& (G\st{g}\rt P) \ar[r] \ar[d]_{\psi} & (Q\oplus P\st{\omega\oplus\id}\rt Q\oplus P) \ar[r] \ar[d]_{\varphi} & (G_1\oplus P\st{l}\rt Q) \ar[d]_{\theta}\ar[r]& 0\\0\ar[r]& (G'\st{g'}\rt P') \ar[r] & (Q'\oplus P'\st{\omega\oplus \id}\rt Q'\oplus P') \ar[r] & (G'_1\oplus P'\st{l'}\rt Q')\ar[r]& 0,}\]in $\mon(\omega, \G)$ with exact rows. One should note that the existence of the morphism $\varphi$ follows from the fact that $ (Q'\oplus P'\st{\omega\oplus \id}\rt Q'\oplus P')$ is a projective (and so, injective) object of $\mon(\omega, \G)$. According to this diagram, we obtain the equality $\syz_{\mon}^{-1}((G\st{g}\rt P)\st{\psi}\lrt (G'\st{g'}\rt P'))=( (G_1\oplus P\st{l}\rt Q)\st{\theta}\lrt (G'_1\oplus P'\st{l'}\rt Q'))$, and then, $\underline{T}_1\syz_{\mon}^{-1}((G\st{g}\rt P)\st{\psi}\lrt (G'\st{g'}\rt P'))=(\cok l\rt\cok l')$. On the other hand, applying the cokernel functor to the latter diagram, gives us the following commutative diagram: \[\xymatrix{0\ar[r]& \cok g \ar[r] \ar[d] & Q/{\omega Q} \ar[r] \ar[d] & \cok l \ar[d]\ar[r]& 0\\0\ar[r]& \cok g' \ar[r] & Q'/{\omega Q'} \ar[r] & \cok l'\ar[r]& 0,}\]in $\G(R)$ with exact rows. Consequently, we get the equalities $$\syz^{-1}_{\G}\underline{T}_1((G\st{g}\rt P)\st{\psi}\lrt (G'\st{g'}\rt P'))=\syz^{-1}_{\G}(\cok g\rt\cok g')=(\cok l\rt\cok l').$$ Hence one may apply \cite[Lemma 2.8, page 23]{happel1988triangulated} and deduce that $\underline{T}_1$ is a triangle functor. So the proof is completed. }

\end{proof}

\begin{theorem}\label{gor} There is a fully faithful triangle functor $T:\umon(\omega, \G)\lrt\ds(R)$ sending each object $(G\st{f}\rt Q)$ to $\cok f$, viewed as a stalk complex. Moreover, $T$ is dense if and only if $R$ (and so, $S$) are Gorenstein rings.
\end{theorem}
\begin{proof}It is known that $T_2:\underline{\G}(R)\rt\ds(R)$ assigning each object to its stalk complex, is a fully faithful triangle functor, see \cite[Theorem 3.1]{bergh2015gorenstein}. This, in conjunction with Proposition \ref{dd}, yields that $T=T_2\circ \underline{T}_1:\umon(\omega, \G)\rt\ds(R)$ sending each object $(G\st{f}\rt P)$ to $\cok f$, as a stalk complex, is a fully faithful triangle functor, giving the first assertion. For the second assertion, assume that $R$ is Gorenstein. {By a fundamental result of Buchweitz and Happel \cite{buchweitz1987maximal}, the triangle functor $T_2$ is equivalence. This together with Proposition \ref{dd}, would imply that $T$ is an equivalence functor}, as well.
Conversely, assume that the functor $T$ is dense. Assume that $k$ is the residue field of $S$. Since $T$ is dense, there is an object $(G\st{f}\rt P)\in\mon(\omega, \G)$ such that  $T(f)=\cok f\cong k$ in $\ds(R)$. Assume that $e:\cok f\lrt k$ is an isomorphism of stalk complexes in $\ds(R)$. So $e$ is represented by a diagram
$$ \xymatrix{ &X^{\bullet} \ar[dl]_{s_1} \ar[dr]^{s_2} \\ \cok f && k } $$of complexes and maps in $\db(R)$, where the mapping cones $\con(s_1)$ and $\con(s_2)$ are prefect complexes. Take the exact triangle $X^{\bullet}\st{s_1}\lrt\cok f\lrt\con(s_1)\lrt X^{\bullet}[1]$ in $\db(R)$. Since $\con(s_1)$, as a complex of $R$-modules, has finite projective dimension and $\pd_SR=1$, its projective dimension over $S$, will be also finite. Combining this with the fact that the Gorenstein projective dimension of the  $S$-module $\cok f$ is at most one, ensures that the Gorenstein projective dimension of $X^{\bullet}$ over $S$ is finite, see \cite[Theorem 3.9(1)]{veliche2006gorenstein}. Now consider the exact triangle $X^{\bullet}\st{s_2}\lrt k\lrt\con(s_2)\lrt X^{\bullet}[1]$. Since $X^{\bullet}$ and $\con(s_2)$, as  complexes of $S$-modules, have finite projective dinension, one may deduce that the Gorenstein projective dimension of $k$ over $S$  is finite. Thus $S$, and so $R$, will be Gorenstien rings, see \cite[(1.4.9)]{christensen2000gorenstein}. The proof then is completed.
\end{proof}

\begin{theorem}\label{reg1} The fully faithful triangle functor $T\circ\underline{i}:\umon(\omega, \cp)\lrt\ds(R)$ is dense if and only if $S$ is a regular ring.
\end{theorem}
\begin{proof} The result indeed follows from the proof of Theorem \ref{gor}, and using the facts that over regular rings, every Gorenstein projective module, is projective, and the finiteness of the projective dimension of the residue field $k$ of $S$, yields that $S$ is regular. Hence, the proof is finished.
\end{proof}
}

{\bf{Acknowledgments.}}
The authors would like to thank the referee for reading the paper very carefully and giving a lot of valuable suggestions kindly and patiently.

\bibliographystyle{siam}

\end{document}